\documentclass[11pt]{article}
\usepackage{amsmath,amssymb,amsthm}
\usepackage{fullpage}
\usepackage[colorlinks=true,linkcolor=blue,citecolor=blue]{hyperref}
\usepackage{seqsplit}
\usepackage{enumitem}

\newtheorem{theorem}{Theorem}[section]
\newtheorem{lemma}[theorem]{Lemma}
\newtheorem{proposition}[theorem]{Proposition}
\newtheorem{corollary}[theorem]{Corollary}

\theoremstyle{definition}
\newtheorem{definition}[theorem]{Definition}

\theoremstyle{definition}
\newtheorem{conjecture}[theorem]{Conjecture}

\theoremstyle{remark}
\newtheorem{remark}[theorem]{Remark}

\newcommand{\C}{\mathcal{C}}
\newcommand{\Q}{\mathbb{Q}}
\newcommand{\Z}{\mathbb{Z}}

\newcommand{\Tr}{\operatorname{Tr}}
\newcommand{\Gal}{\operatorname{Gal}}

\title{A trace lower bound for totally positive units in $\mathbb{Q}(\zeta_p)^+$\\ and its application to fusion category}
\author{Yue Meng and Zhiqiang Yu}

\begin{document}
\date{}

\maketitle
\begin{abstract}
Let $p = 4q + 1$ be a prime integer, where $q$ is an odd prime.
Let $L = \mathbb{Q}(\zeta_p)^+$ denote the maximal totally real subfield of the $p$-th cyclotomic field $\mathbb{Q}(\zeta_p)$. In this paper,  we prove that if $U_L^+=U_L^2$ (i.e., every totally positive unit is a square in $L$), then for any non-trivial subfield $K \subseteq L$ and any non-trivial totally positive unit $\varepsilon \in U_K^+$, the absolute trace satisfies the sharp lower bound $\Tr_{K/\mathbb{Q}}(\varepsilon) \ge p - 4$. As an application, we show that   a spherical fusion category $\C$ with global dimension $\dim(\mathcal{C})=p=4q+1$ is pointed, where  $2$ is a primitive root modulo $q$, or $p=4q+1<2^{22}$ and $p\neq 29$.
\end{abstract}

{\small\quad\textbf{Keywords}:  Spherical fusion category; Global dimension; Totally positive unit}
\section{Introduction}
A fusion category over the complex field $\mathbb{C}$  is a semisimple finite tensor category. Fusion categories have two important concepts: the Frobenius-Perron dimension and global (or categorical) dimension.  Structure and properties of fusion categories  with restrictions on their Frobenius-Perron dimension   are widely studied, see \cite{ENO2005,EGNO}. However, unlike the   Frobenius-Perron dimension, currently, we even don’t know the structures of fusion categories with a prime global dimension.

Recently,  progress was made in topics related to the  global dimension of fusion category, see \cite{Ostrik2015,Schopieray2021,Schopieray2026,Yu2023} and the references therein for details.  It was asked in \cite[Question 3.8]{Yu2022}  whether a spherical braided fusion category of prime global dimension is pointed (i.e., all simple objects have Frobenius-Perron dimension $1$) or equivalent to $\text{Fib}\boxtimes\text{Fib}^\sigma$, where $\text{Fib}$ and $\text{Fib}^\sigma$ are  the Fibonacci fusion category and its Galois conjugate with global dimensions being $\frac{5+\sqrt{5}}{2}$ and $\frac{5-\sqrt{5}}{2}$, respectively, which was proved in \cite{Schopieray2021}. 
Moreover,  spherical fusion categories of prime dimension $p=2q+1$ are   pointed if $q$ is also an odd prime \cite[Theorem 5.5]{Schopieray2021}. Meanwhile, non-simple modular fusion categories of global  dimension $p^2$ are completely classified in \cite[Theorem 4.14]{Yu2023}.

Let $\C$ be a spherical fusion category of global dimension $p$, where $p$ is prime. One of the main techniques used in \cite[Theorem 5.5]{Schopieray2021} is by showing the Frobenius-Perron dimension (generally, formal codegrees \cite{Ostrik2009}) of $\C$ equals to $p\cdot u$ for some algebraic unit $u\in\mathbb{Q}(\zeta_p)$ (see \cite[Lemma 5.1]{Schopieray2021}).  Together with the Siegel's trace theorem \cite[Theorem III]{Siegel1945},  the relation of formal codegrees   \cite[Proposition 2.10]{Ostrik2015} shows that $\C$ must be pointed if $p=2q+1$, where $q$ is an odd prime. 
Inspired by the result of \cite{Schopieray2021}, in this paper, we continue to consider the structure of $\C$ for the case $p=4q+1$, where $q$ is also an odd prime. 
However,  it turns out that the Siegel's trace bound is insufficient to show the pointedness of $\C$ for  $p=4q+1$, see \cite[Example 5.7]{Schopieray2021} and Lemma \ref{13ranklemma} below. Notice that the formal codegrees of fusion categories are totally positive \cite{Ostrik2009}, so we first  pay our  attentions to give a new lower bound for totally positive algebraic units in $\mathbb{Q}(\zeta_p)$. 
Namely,  we consider the trace of totally positive units $\varepsilon \in U_K^+$   for any subfield $\mathbb{Q} \subsetneq K \subseteq L:=\mathbb{Q}(\zeta_p)^+$.  

Our first main result is stated as follows.
\begin{theorem}\label{maintheorem1}
Let $p = 4q + 1$ be a prime integer, where $q$ is an odd prime integer. Let $L = \mathbb{Q}(\zeta_p)^+$, and let $K \neq \mathbb{Q}$ be any totally real subfield of $L$. If $U_L^+=U_L^2$ (i.e., every totally positive unit is a square in $L$), then for any non-trivial totally positive unit $1 \neq \varepsilon \in U_K^+$, we have
\begin{equation*}
\Tr_{K/\mathbb{Q}}(\varepsilon) \ge p - 4.
\end{equation*}
Furthermore, this lower bound is sharp and is attained in the subfield of degree $q$.
\end{theorem}

The following corollary is a direct result of \cite[Proposition 1] {Dummit2020}, \cite[Theorem 4.4] {KimLim2007} and \cite[Proposition 3.1]{FujimaIchimura2018}. 
\begin{corollary}\label{1-2}
Let $p = 4q + 1$ be a prime integer, where $q$ is an odd prime integer. Let $L = \mathbb{Q}(\zeta_p)^+$, and let $K \neq \mathbb{Q}$ be any totally real subfield of $L$. If 
\begin{itemize}
    \item $2$ is a primitive root modulo $q$, or
    \item $p=4q+1<2^{22}$ and $p\neq 29$,
\end{itemize}
then for any non-trivial totally positive unit $1 \neq \varepsilon \in U_K^+$, we have
$
\Tr_{K/\mathbb{Q}}(\varepsilon) \ge p - 4.
$
\end{corollary}

 As an application of this trace inequality to
fusion categories, we have the following theorem.
\begin{theorem}\label{maintheorem2}
Let $p = 4q + 1$ be a prime integer, where $q$ is an odd prime integer. Let  $\mathcal{C}$ be a spherical fusion category with $\dim(\mathcal{C})=p$.  If 
$2$ is a primitive root modulo $q$, or
 $p=4q+1<2^{22}$ and $p\neq 29$, then $\mathcal{C}$ is pointed.
\end{theorem}

The paper is organized as follows. In section \ref{sectionpreliminary}, we review some basic notions of algebraic number theory and fusion category. In section \ref{sectiontracebound}, for primes $p=4q+1$, we first prove the trace inequality under the assumption that every totally positive unit is a square in $L=\mathbb{Q}(\zeta_p)^+$ (Theorem \ref{maintheorem1}), and then we show that the trace inequality always holds under certain restrictions (Corollary \ref{1-2}). In section \ref{sectionfusion}, we apply to trace the bound to classify spherical fusion category of prime global dimension (Theorem \ref{maintheorem2}).
\section{Preliminaries}\label{sectionpreliminary}
In this section, we recall some basic notions and properties of algebraic number theory, fusion categories.
For further details, readers are referred to \cite{AlacaWilliams2003}, \cite{ENO2005}, \cite{EGNO}, \cite{Shimura2010}, \cite{washington1997}.
\subsection{Cyclotomic fields,   totally positive units}
In this subsection, we recall some elementary results about algebraic number theory, and the following notations will be used throughout this paper. Let \(\overline{\mathbb{Q}}\) denote the algebraic closure of \(\mathbb{Q}\) in \(\mathbb{C}\), any subfield \(K\) of \(\overline{\mathbb{Q}}\) is called an algebraic number field. In this paper, we always assume that $K$ is finite Galois extension of $\mathbb{Q}$.  Let $G=\operatorname{Gal}(K/\mathbb{Q})$.  Then for $\alpha \in K$, we have 
\begin{equation*}
N_{K/\mathbb{Q}}(\alpha)=\prod_{\sigma \in G}\sigma(\alpha), \qquad  \Tr_{K/\mathbb{Q}}(\alpha)=\sum_{\sigma \in G}\sigma(\alpha).
\end{equation*}

Let $K$ be an algebraic number field. An algebraic integer is an element $\alpha\in\mathbb{C}$ which is a root of a monic polynomial with coefficients in $\mathbb{Z}$, we use $\mathcal{O}_K$ denote the ring of all algebraic integers in $K$,  $\mathcal{O}_K$ is a free $\mathbb{Z}$-module of rank $[K:\mathbb{Q}]$. Let $U_K$ denotes the group of algebraic units of $K$, then we have $U_K:=\mathcal{O}_K^{\times}=\{\alpha \in \mathcal{O}_K \mid N_{K/\mathbb{Q}}(\alpha)=\pm1 \}$.

Given a positive integer $n\geq 3$, let $\zeta_n=\exp(2\pi i/n)$ be a primitive $n$-th root of unity. The field $\mathbb{Q}(\zeta_n)$ is called the $n$-th cyclotomic field, which is the main research interest to  the theory of fusion category \cite[Theorem 8.51]{ENO2005}\cite{Ostrik2015}. The Galois  group $\Gal(\mathbb{Q}(\zeta_n)/\mathbb{Q})\cong(\mathbb{Z}/n\mathbb{Z})^\times$, and the maximal real subfield of $\mathbb{Q}(\zeta_n)$ is $\mathbb{Q}(\zeta_n)^+:=\mathbb{Q}(\zeta_n+\zeta_n^{-1})$ and $[\mathbb{Q}(\zeta_n):\mathbb{Q}(\zeta_n)^+]=2$. It follows from Galois theory that $\Gal(\mathbb{Q}(\zeta_n)^+/\mathbb{Q})\cong (\mathbb{Z}/n\mathbb{Z})^\times/\{\pm 1\}$. In the following, we restrict our attention to the case where 
$n$ is prime. If $n=p$ is a prime,  the group of cyclotomic units is the subgroup of \(U_{\mathbb{Q}(\zeta_p)}=\mathcal{O}_{\mathbb{Q}(\zeta_p)}^{\times}\) generated by \(\zeta_p\) and the elements
$
\frac{\zeta_p^{a}-1}{\zeta_p-1}$, for
$ a\in\mathbb{Z}$ and $p\nmid a$ \cite[Lemma 8.1]{washington1997}. 

We restrict to totally real algebraic number fields $K/\mathbb{Q}$ of degree $n$, so that all embeddings are real and traces are ordinary sums of conjugates.  
Let $\sigma_1,\dots,\sigma_n: K \hookrightarrow \mathbb{R}$ be the real embeddings.  
By Dirichlet’s unit theorem \cite[Theorem 13.1.1]{AlacaWilliams2003}, the unit group of the ring of algebraic integers $\mathcal{O}_K$ has the structure
\begin{equation*}
U _K:= \mathcal{O}_K^\times \cong \{\pm 1\}\times \mathbb{Z}^{n-1}.
\end{equation*}
The group of totally positive units is \(U_K^{+}: = \{u \in U _K \mid \sigma(u) > 0\) for all real embeddings \(\sigma \colon K \hookrightarrow \mathbb{R}\)\}. Let $U_K^2$ the subgroup of
the unit squares of $K$. It is easy to see that $U_K^2 \subseteq U_K^{+}$.

\subsection{Fusion categories}
A fusion category over $\mathbb{C}$ is a semisimple finite  tensor category with simple tensor unit $\mathbf{1}$. Let $\mathcal{O}(\C)$ be  the set of isomorphism classes of simple objects of a fusion category $\C$, the cardinality of $\mathcal{O}(\C)$ is called the rank of $\C$, and denoted by $\mathrm{rank}(\C)$. 
And let $\mathrm{Gr}(\mathcal{C})$ denote the Grothendieck ring of $\mathcal{C}$.
Then the Frobenius-Perron homomorphism $\mathrm{FPdim}(-):\mathrm{Gr}(\C)\to\mathbb{C}$ is the unique ring homomorphism such that $\mathrm{FPdim}(X)\geq1$ for all objects $X$ \cite{EGNO}. Moreover,  $\mathrm{FPdim}(X)$ is an algebraic integer and is defined as the Frobenius-Perron dimension of the object $X$. The sum
\begin{equation*}
\mathrm{FPdim}(\C):=\sum_{X\in\mathcal{O}(\C)}\mathrm{FPdim}(X)^2
\end{equation*}
is called  the Frobenius-Perron dimension of $\C$. A fusion category is called pointed if every simple object is invertible, equivalently, $\C$ is tensor equivalent to the category   $\text{Vec}_G^\omega$ of $G$-graded finite-dimensional vector spaces over $\mathbb{C}$, where $G$ is a finite group and $\omega\in Z^3(G,\mathbb{C}^*)$

Let $\C$ be a pivotal fusion category with a pivotal structure $j$, that is, $j$ is a natural isomorphism from identity tensor functor $\text{id}_\C$ to the double dual tensor functor $(-)^{**}$.  Then the quantum dimension of an object $X$ is defined by the (categorical) trace of $\text{id}_X$, that is,
\begin{equation*}
\dim_j(X)=\text{Tr}(\text{id}_X):=\text{ev}_X\circ(j_X\circ \text{id}_X)\otimes \text{id}_{X^*}\circ\text{coev}_X,
\end{equation*}
here we suppress the associativity and   unit constraints  of $\C$.  $\C$ is a spherical  fusion category  if $\dim_j(X)=\dim_j(X^*)$ for all objects $X$ of $\C$ \cite{EGNO}. Then we define the global (or, categorical) dimension of the spherical fusion category $\C$ as
\begin{equation*}
\dim(\C):=\sum_{X\in\mathcal{O}(\C)}\dim_j(X)^2.
\end{equation*}
Notice that $\dim_j(-)$ induces a homomorphism from $\mathrm{Gr}(\C)$ to $\mathbb{C}$
\cite[Proposition 4.7.12]{EGNO}. When \(\dim(\mathcal{C}) = \operatorname{FPdim}(\mathcal{C})\), we say that \(\mathcal{C}\) is pseudo-unitary.

Let $\C$ be a fusion category. Then the complexified ring $\mathrm{Gr}(\C)\otimes_\Z\mathbb{C}$ is semisimple. Given an irreducible representation $\chi$ of $\mathrm{Gr}(\C)\otimes_\Z\mathbb{C}$,  the element
\begin{equation*}
\alpha_\chi=\sum_{X\in\mathcal{O}(\C)}\text{tr}_\chi(X)X^*
\end{equation*}
is central, where $\text{tr}_\chi$ is the ordinary trace function on the representation $\chi$. Moreover, $\chi'(\alpha_\chi)=0$ if $\chi\ncong\chi'$ and $f_\chi:=\chi(\alpha_\chi)$ is a positive algebraic integer \cite{Lusztig}, $f_\chi$ is called a formal codegree of fusion category $\C$ \cite{Ostrik2009}. For example, $\mathrm{FPdim}(\C)$ and $\dim(\C)$ are formal codegrees of $\C$.
It was shown that $\frac{\dim(\C)}{f_\chi}$ is always a totally positive algebraic integer
\cite[Corollary 2.14]{Ostrik2015} and the set of formal codegrees of $\C$ satisfy the following equation \cite[Proposition 2.10]{Ostrik2015}
\begin{align}\label{generalclasseq}
\sum_{\chi}\frac{\chi(1)}{f_\chi}=1,
\end{align}
which plays a key role in the classification of spherical fusion categories, see \cite{Schopieray2021,Yu2023}.
\section{Trace lower bound for totally positive units in $\mathbb{Q}(\zeta_p)^+$}\label{sectiontracebound}
Let $q$ be an odd prime and $p = 4q+1$ be a prime, and $q \ge 3$ is prime. For simplification,  denote $L: = \mathbb{Q}(\zeta_p)^+$.
We begin by detailing the Galois extension $L/\mathbb{Q}$ and the structure of totally positive unit groups in its subfields.
Since $G = \Gal(L/\mathbb{Q}) \cong (\Z/p\Z)^\times / \{\pm 1\}$ is a cyclic group of order $2q$, the nontrivial subfields $K \neq \mathbb{Q}$ fall precisely into the following three cases:
\begin{itemize}
    \item $n = 2q$: $K = L = \mathbb{Q}(\zeta_p)^+$, with relative degree $m = [L:K] = 1$.
    \item $n = q$: $K = K_q$, with relative degree $m = [L:K] = 2$.
    \item $n = 2$: $K = \mathbb{Q}(\sqrt{p})$, with relative degree $m = [L:K] = q$.
\end{itemize}

\subsection{Total positivity and square units}

\begin{theorem}\label{squares}
Let $q$ be an odd prime and $p = 4q+1$ be a prime. Let $L = \mathbb{Q}(\zeta_p)^+$, and let $K$ be the unique subfield of $L$ with degree $[K:\mathbb{Q}] = q$. If $U_L^+=U_L^2$, then every totally positive unit of $K$ is a square in $K$, i.e.,
$
U_K^+ = U_K^2.
$
\end{theorem}

\begin{proof}
Let $u \in U_K^+$ be an arbitrary totally positive unit. Since \(K\subset L\) and both fields are totally real, we have
$
U_K^+\subseteq U_L^+.
$
Thus \(u\in U_L^+\).
By hypothesis, every totally positive unit of $L$ is a square in $L$, so there exists a unit $v \in U_L$ such that $v^2 = u$.  We claim that $v$ must belong to $K$.

Suppose $v \notin K$. Since $\Gal(L/\mathbb{Q}) \cong \mathbb{Z}/2\mathbb{Z} \times \mathbb{Z}/q\mathbb{Z}$ with $q$ being an odd prime, the relative extension $L/K$ is quadratic. Thus, $v \notin K$ implies that $L = K(v) = K(\sqrt{u})$.

On the other hand, since $p = 4q+1 \equiv 1 \pmod 4$, the unique quadratic subfield of $\mathbb{Q}(\zeta_p)$ is $\mathbb{Q}(\sqrt{p})$. Since $[K:\mathbb{Q}] = q$ is odd, $\mathbb{Q}(\sqrt{p}) \cap K = \mathbb{Q}$. Furthermore, the compositum $K(\sqrt{p}) \subseteq L$ has degree $[K(\sqrt{p}):\mathbb{Q}] = [K:\mathbb{Q}] [\mathbb{Q}(\sqrt{p}):\mathbb{Q}] = 2q$. Comparing degrees over $\mathbb{Q}$ with $[L:\mathbb{Q}] = (p-1)/2 = 2q$, we obtain $L = K(\sqrt{p})$.

Since two quadratic extensions $K(\sqrt{u})$ and $K(\sqrt{p})$ coincide if and only if $u/p$ is a square in $K^\times$. Therefore, there exists some element $w \in K^\times$ such that
$
u = p \cdot w^2.
$
Taking the norm $N_{K/\mathbb{Q}}$ on both sides:
$
N_{K/\mathbb{Q}}(u) = N_{K/\mathbb{Q}}(p \cdot w^2) = p^q \cdot \left(N_{K/\mathbb{Q}}(w)\right)^2.
$
Since $u \in U_K^+$, $N_{K/\mathbb{Q}}(u) = 1$. Substituting into the norm equation, we get
$
1 = p^q \cdot \left(N_{K/\mathbb{Q}}(w)\right)^2.
$
Let $N_{K/\mathbb{Q}}(w) = \frac{a}{b} \in \mathbb{Q}^\times$ with $\gcd(a,b)=1$ for $a, b \in \mathbb{Z} \setminus \{0\}$. Clearing denominators gives
\begin{equation}\label{b2}
b^2 = p^q \cdot a^2.
\end{equation}
We use the $p$-adic valuation $v_p$ \cite[Chapter II 6.1]{Shimura2010} on both sides of Equation \eqref{b2}:
\begin{itemize}
    \item For the left-hand side: $v_p(b^2) = 2 v_p(b)$, which is an even integer.
    \item For the right-hand side: $v_p(p^q a^2) = q + 2 v_p(a)$, which is an odd integer.
\end{itemize}
This asserts that an even integer equals an odd integer, which is a contradiction.

Hence, our assumption that $v \notin K$ is false, forcing $v \in K$. Since $v \in U_K$ and $v^2 = u$, we conclude that $u \in U_K^2$, which completes the proof that $U_K^+ = U_K^2$.
\end{proof}

As defined in Theorem \ref{squares}, $K$ be the unique subfield of $L$ with degree $[K:\mathbb{Q}] = q$. If $U_L^+=U_L^2$, we can also express any totally positive unit $\varepsilon \in U_K^+$ as $\varepsilon = \eta^2$ for some unit $\eta \in U_K$. Consequently,
$
\Tr_{K/\mathbb{Q}}(\varepsilon) = \Tr_{K/\mathbb{Q}}(\eta^2).
$

\subsection{The trace quadratic form on $\mathbb{Q}(\zeta_p)^+$}
Let \(p = 4q + 1\) be an odd prime, where $q \ge 3$ is a prime integer. Define
\begin{equation*}
e_j := \zeta_p^j + \zeta_p^{-j}, \quad 1 \le j \le 2q.
\end{equation*}
By \cite[Corollary 2.2.1]{Davis1969},
$
\mathcal{O}_L = \bigoplus_{j=1}^{2q} \mathbb{Z} e_j.
$
Thus \(\{ e_1, e_2, \dots, e_{2q} \}\) is a \(\mathbb{Z}\)-basis of \(\mathcal{O}_L\). Note that $\sum_{j=1}^{2q} e_j = -1$.

\begin{lemma}\label{trace_matrix_lemma}
For any $1 \le j, k \le 2q$, the trace pairing $\Tr_{L/\mathbb{Q}}(e_j e_k)$ satisfies:
\begin{equation*}
\Tr_{L/\mathbb{Q}}(e_j e_k) = 
\begin{cases} 
p - 2, & \text{if } j = k, \\
-2, & \text{if } j \neq k.
\end{cases}
\end{equation*}
\end{lemma}

\begin{proof}
Using the  identity $e_j e_k = e_{j+k} + e_{j-k}$, we evaluate the trace.
 When $j = k$, $e_j^2 = e_{2j} + 2$. Since $1 \le 2j < p$,  $\Tr_{L/\mathbb{Q}}(e_{2j})= \sum_{j=1}^{2q} e_j = -1$. Hence $\Tr_{L/\mathbb{Q}}(e_j^2) = -1 + 2 [L:\mathbb{Q}] = -1 + 4q = p - 2$.
 When $j \neq k$,  $e_j e_k = e_{j+k} + e_{j-k}$. Since $1 \le |j \pm k| < p$ and $j \neq k$,  $\Tr_{L/\mathbb{Q}}(e_j e_k) = (-1) + (-1) = -2$.
\end{proof}

\begin{proposition}\label{trace_quadratic_prop}
Let $x = \sum_{j=1}^{2q} c_j e_j \in \mathcal{O}_L$ with $c_j \in \Z$. Then
\begin{equation}\label{trace_L_formula}
\Tr_{L/\mathbb{Q}}(x^2) = p \sum_{j=1}^{2q} c_j^2 - 2 \left( \sum_{j=1}^{2q} c_j \right)^2.
\end{equation}
\end{proposition}
\begin{proof}
Expanding $\Tr_{L/\mathbb{Q}}(x^2)$ and using Lemma~\ref{trace_matrix_lemma}:
\begin{align*}
\Tr_{L/\mathbb{Q}}(x^2) &= \sum_{j=1}^{2q} c_j^2 \Tr_{L/\mathbb{Q}}(e_j^2) + \sum_{j \neq k} c_j c_k \Tr_{L/\mathbb{Q}}(e_j e_k) 
= (p-2) \sum_{j=1}^{2q} c_j^2 - 2 \sum_{j \neq k} c_j c_k \\
&= p \sum_{j=1}^{2q} c_j^2 - 2 \left( \sum_{j=1}^{2q} c_j^2 + \sum_{j \neq k} c_j c_k \right) 
= p \sum_{j=1}^{2q} c_j^2 - 2 \left( \sum_{j=1}^{2q} c_j \right)^2.
\end{align*}
\end{proof}

\subsection{Proof of the Main Theorem}
Since $L = \mathbb{Q}(\zeta_p)^+$, $L/\mathbb{Q}$ is a Galois extension of degree $[L:\mathbb{Q}] = 2q$ with cyclic Galois group $G = \operatorname{Gal}(L/\mathbb{Q}) \cong \mathbb{Z}/2q\mathbb{Z}$. By the Fundamental Theorem of Galois Theory, any intermediate subfield $K \subseteq L$ of degree $[K:\mathbb{Q}] = n$ corresponds to a unique subgroup $H = \operatorname{Gal}(L/K) \le G$ of order $|H| = m = \frac{2q}{n}$. The subfield $K$ is precisely the fixed field of $H$, i.e.,
\begin{equation*}
K = L^H = \{ \alpha \in L \mid \sigma(\alpha) = \alpha \text{ for all } \sigma \in H \}.
\end{equation*}

Let $\{e_1, e_2, \dots, e_{2q}\}$ be a $\mathbb{Z}$-basis of the ring of integers $\mathcal{O}_L$ upon which the Galois group $G$ acts as a group of permutations on the index set $I = \{1, 2, \dots, 2q\}$. That is, for each $\sigma \in G$, there exists a unique permutation $\pi_\sigma \in S_{2q}$ such that:
\begin{equation*}
\sigma(e_j) = e_{\pi_\sigma(j)} \quad \text{for all } j \in \{1,2, \dots, 2q\}.
\end{equation*}

\begin{lemma}\label{fix}
The action of $G$ on the index set $I$ is fixed-point-free for non-identity elements. That is, if $\sigma \in G$ is non-trivial, then $\pi_\sigma(j) \neq j$ for all $j \in I$.
\end{lemma}

\begin{proof}
 For any $\sigma_a \in G \cong (\mathbb{Z}/p\mathbb{Z})^\times / \{\pm 1\}$ and notice that $e_j=\zeta_p^j + \zeta_p^{-j} = 2\cos\frac{2\pi j}{p}$, the condition $\sigma_a(e_j) = e_j$ is equivalent to $\cos(\frac{2\pi aj}{p}) = \cos(\frac{2\pi j}{p})$, which yields $aj \equiv \pm j \pmod p$. Since $\gcd(j, p) = 1$, this forces $a \equiv \pm 1 \pmod p$, meaning $\sigma_a = \mathrm{id}$. Thus, $\mathrm{Stab}_G(e_j) = \{\mathrm{id}\}$, so $e_j$ has $2q$ distinct conjugates under $G$ and is a primitive element for $L/\mathbb{Q}$.
\end{proof}

\begin{proposition}
Let $H \le G$ be a subgroup of order $|H| = m$. Under the restricted action of $H$ on $I$, the index set $I$ decomposes into a disjoint union of $n = \frac{2q}{m}$ equicardinal orbits:
\begin{equation*}
I = \mathcal{O}_1 \sqcup \mathcal{O}_2 \sqcup \dots \sqcup \mathcal{O}_n,
\end{equation*}
where each orbit $\mathcal{O}_i$ has cardinality $|\mathcal{O}_i| = |H| = m$. Furthermore, $j, k \in I$ belong to the same orbit $\mathcal{O}_i$ if and only if there exists some $\sigma \in H$ such that $k = \pi_\sigma(j)$.
\end{proposition}

\begin{proof}
By Lemma \ref{fix}, the stabilizer subgroup $\mathrm{Stab}_H(j)$ is trivial for every $j \in I$. By the Orbit-Stabilizer Theorem, each orbit has cardinality $|\mathcal{O}_j| = |H| / |\mathrm{Stab}_H(j)| = m$. Since the orbits partition $I$, the number of orbits is $n = |I| / m = 2q / m$.
\end{proof}

Since 
$
\mathcal{O}_L = \bigoplus_{j=1}^{2q} \mathbb{Z} e_j
$ \cite[Corollary 2.2.1]{Davis1969}, for any   $\eta \in \mathcal{O}_K$,  $\eta$ can be uniquely expressed as $ \eta=\sum_{j=1}^{2q} c_j e_j$ with $c_j \in \mathbb{Z}$. Then we  have 
\begin{equation*}
\eta = \sigma(\eta) = \sum_{j=1}^{2q} c_j \sigma(e_j) = \sum_{j=1}^{2q} c_j e_{\pi_\sigma(j)} = \sum_{k=1}^{2q} c_{\pi_\sigma^{-1}(k)} e_k,\quad \forall \sigma\in H.
\end{equation*}
Hence 
$
c_k = c_{\pi_\sigma^{-1}(k)}\ \text{for all } \sigma \in H \text{ and all } k \in \{1, \dots, 2q\}
$, that is,  $c_j = c_k$ whenever $j$ and $k$ belong to the same $H$-orbit $\mathcal{O}_i$. Therefore, the coefficient $c_j$ is constant on each orbit $\mathcal{O}_i$.

Setting $ d_i:=c_j$ for all $j \in \mathcal{O}_i$, where $d_1, \dots, d_n \in \mathbb{Z}$. Then we have
\begin{equation*}
\eta = \sum_{i=1}^n \sum_{j \in \mathcal{O}_i} c_j e_j = \sum_{i=1}^n d_i \left( \sum_{j \in \mathcal{O}_i} e_j \right) = \sum_{i=1}^n d_i E_i,
\end{equation*}
where $E_i = \sum_{j \in \mathcal{O}_i} e_j$ (for $i=1,\dots,n$) are the Gauss periods  for the subfield $K/\mathbb{Q}$, each consisting of a sum of $m$ basis elements.

Define
$
T := \sum_{i=1}^n d_i$ and $\bar{d} := \frac{T}{n}.
$
Then we have $\sum_{j=1}^{2q} c_j^2 = m \sum_{i=1}^n d_i^2$ and $\sum_{j=1}^{2q} c_j = m T$.
Using the transitivity of trace $\Tr_{L/\mathbb{Q}}(\eta^2) = m \Tr_{K/\mathbb{Q}}(\eta^2)$ and Equation \eqref{trace_L_formula}, we know
\begin{equation*}
\Tr_{K/\mathbb{Q}}(\eta^2) = p \sum_{i=1}^n d_i^2 - 2m T^2.
\end{equation*}
Since $2m = \frac{4q}{n} = \frac{p-1}{n}$, by using $\sum_{i=1}^n d_i^2 = \sum_{i=1}^n (d_i - \bar{d})^2 + n \bar{d}^2 = \sum_{i=1}^n (d_i - \bar{d})^2 + \frac{T^2}{n}$, we have

\begin{proposition}\label{subfield_trace_formula}
For any element $\eta = \sum_{i=1}^n d_i E_i \in \mathcal{O}_K$, where $E_i$ are the orbit sums of $e_j$,
\begin{equation}\label{master_trace_eq}
\Tr_{K/\mathbb{Q}}(\eta^2) = p \sum_{i=1}^n (d_i - \bar{d})^2 + \frac{T^2}{n}.
\end{equation}
\end{proposition}

We now complete the proof of Theorem~\ref{maintheorem1} by analyzing the three subfield cases.

\begin{proof}[Proof of Theorem \ref{maintheorem1}]
Let $\varepsilon = \eta^2 \in U_K^+$ with $\varepsilon \neq 1$, which implies $\eta \neq \pm 1$. Thus, the integer vector $(d_1, \dots, d_n) \in \Z^n$ cannot have all equal components, so the discrete variance $V = \sum_{i=1}^n (d_i - \bar{d})^2 > 0$. The minimum non-zero value of $V$ over $\Z^n$ is $\frac{n-1}{n}$, attained when $n-1$ components equal $a$ and one component equals $a+1$.

\medskip
\noindent\textbf{Case 1: $n = q$ (The $q$-th degree subfield $K_q$, $m=2$).} \\
In this case, $2m = 4$. Equation \eqref{master_trace_eq} becomes:
$\Tr_{K_q/\mathbb{Q}}(\eta^2) = p \sum_{i=1}^q (d_i - \bar{d})^2 + \frac{T^2}{q}.$
\begin{itemize}
     \item If $T \neq 0$, since $T \in \Z$, $T^2 \ge 1$. Combined with $V \ge \frac{q-1}{q}$, we get
    $$\Tr_{K_q/\mathbb{Q}}(\eta^2) \ge p \left( \frac{q-1}{q} \right) + \frac{1}{q} = \frac{p(q-1)+1}{q} = p - \frac{p-1}{q} = p - 4.$$
    
     \item If $T = 0$,  since not all $d_i$ are equal and $\sum d_i = 0$, there must be at least one positive and one negative integer. Thus $\sum d_i^2 \ge 1^2 + (-1)^2 = 2$. Then
     $
    \Tr_{K_q/\mathbb{Q}}(\eta^2) = p \sum_{i=1}^q d_i^2 \ge 2p > p - 4.
    $
\end{itemize}    
\medskip
\noindent\textbf{Case 2: $n = 2q$ (The full field $L$, $m=1$).} 
\begin{itemize}
    \item  If $T \neq 0$,  $\Tr_{L/\mathbb{Q}}(\eta^2) \ge p \left(\frac{2q-1}{2q}\right) + \frac{1}{2q} = p - \frac{p-1}{2q} = p - 2 > p - 4$.
    \item If $T = 0$,  $\Tr_{L/\mathbb{Q}}(\eta^2) \ge 2p > p - 4$.
\end{itemize}
\medskip
\noindent\textbf{Case 3: $n = 2$ (The quadratic subfield $K = \mathbb{Q}(\sqrt{p})$, $m=q$).} \\
Let $\varepsilon \in U_K^+$ be a non-trivial totally positive unit. For a prime $p \equiv 1 \pmod 4$, the fundamental unit of $K = \mathbb{Q}(\sqrt{p})$ has norm $-1$ \cite[Theorem 11.5.4]{AlacaWilliams2003}. Thus, any totally positive unit must be a square, so $\varepsilon = \eta^2$ for some unit $\eta \in \mathcal{O}_K^\times$.
Since $\mathcal{O}_K = \mathbb{Z}\left[\frac{1+\sqrt{p}}{2}\right]$, we can write $\eta = \frac{t + k\sqrt{p}}{2}$ for integers $t, k$ with $t \equiv k \pmod 2$. If $k = 0$, then $\eta \in \mathbb{Z}$, forcing $\eta = \pm 1$ and $\varepsilon = 1$, which contradicts that $\varepsilon$ is non-trivial. Thus, $k \neq 0$.

The unit condition $N_{K/\mathbb{Q}}(\eta) = \pm 1$ yields the Pell equation $t^2 - p k^2 = \pm 4$. Computing the trace of $\varepsilon = \eta^2$ gives:
$
\Tr_{K/\mathbb{Q}}(\varepsilon) = \Tr_{K/\mathbb{Q}}\left(\frac{t^2 + p k^2 + 2t k\sqrt{p}}{4}\right) = \frac{t^2 + p k^2}{2}.
$
Substituting $t^2 = p k^2 \pm 4$, we obtain $\Tr_{K/\mathbb{Q}}(\varepsilon) = p k^2 \pm 2$. Since $k^2 \ge 1$, it  follows that:
$
\Tr_{K/\mathbb{Q}}(\varepsilon) \ge p - 2 > p - 4.
$

Combining all three cases proves $\Tr_{K/\mathbb{Q}}(\varepsilon) \ge p - 4$ for all non-trivial subfields $K \neq \mathbb{Q}$. The fact that this lower bound is sharp and is attained in the subfield $K_q$  will be proved in Theorem \ref{p-4}. This completes the proof of the theorem.
\end{proof}

Next, we give a proof of Corollary \ref{1-2}.
\begin{proof}[Proof of Corollary \ref{1-2}]
If \( 2 \) is a primitive root modulo \( q \), then 
$
[U_L^+ : U_L^2] = 1
$
by \cite[Theorem 4.4]{KimLim2007}.
If  $p=4q+1<2^{22}$ and $p\neq 29$, then   \cite[Proposition 1]{Dummit2020} states that if the relative class number of $\mathbb{Q}(\zeta_p)$ is odd then $[U_L^+ : U_L^2] = 1$, which is guaranteed by  \cite[Proposition 3.1]{FujimaIchimura2018}.  Thus, the conclusion of Corollary \ref{1-2} follows from that of Theorem \ref{maintheorem1}.
\end{proof}

\begin{remark}\label{remarkp=4q+1}
Let $p=4q+1$ be a prime, where $q$ is an odd prime. It is expected that there exist infinitely many primes of this form, but this has not yet been proven \cite{BatemanHorn1962}. 
For $p<2^{22}$, there are $7733$ primes of the form $4q+1$, where $q$ is an odd prime.

Assume $p<1000$. Then 
\seqsplit{
13,\ 53,\ 149,\ 269,\ 557,\ 653}  are primes such that $2$ is a primitive root
mod $q$;   
\seqsplit{
29,\ 173,\ 293,\ 317,\ 389,\ 509,\ 773,\ 797} are primes for which $2$ is not a primitive root mod $q$. When $p=29$, $U_{\mathbb{Q}(\zeta_{29})^+}^+\neq U_{\mathbb{Q}(\zeta_{29})^+}^2$ as shown in \cite{Dummit2020}   , and it is informed by A. Schopieray that there exist totally positive algebraic units in the subfield of $\mathbb{Q}(\zeta_{29})^+$, which does not satisfy the lower bound given in Theorem \ref{maintheorem1}, their minimal polynomials are 
\begin{align*}
 T^7&-16T^6+89T^5-212T^4+236T^3-119T^2+23T-1,\\
 T^7&-23T^6+119T^5-236T^4+212T^3-89T^2+16T-1,\\
 T^7&-20T^6+130T^5-318T^4+341T^3-162T^2+30T-1,\\
 T^7&-20T^6+130T^5-376T^4+515T^3-307T^2+59T-1,
\end{align*}
we do not know whether these are all such minimal polynomials.
\end{remark}
We conclude  this subsection by presenting the following conjecture.

\begin{conjecture}
Let $p = 4q + 1$ be a prime integer, where $q\neq 7$ is an odd prime integer. Let $L = \mathbb{Q}(\zeta_p)^+$, and let $K \neq \mathbb{Q}$ be any totally real subfield of $L$. For any non-trivial totally positive unit $1 \neq \varepsilon \in U_K^+$, we have
$
\Tr_{K/\mathbb{Q}}(\varepsilon) \ge p - 4.
$
Furthermore, this lower bound is sharp and is reached in the subfield of degree $q$.
\end{conjecture}

\subsection{Explicit construction of sharp cyclotomic units}
In this subsection, we explicitly construct non-trivial totally positive units in $K_q$ whose traces  reach the lower bound $p-4$, which proves that bounds derived in the previous subsections are sharp.

Recall that $p = 4q + 1$, $G = \Gal(L/\mathbb{Q}) \cong (\mathbb{Z}/p\mathbb{Z})^\times / \{\pm 1\}$. Fix $a \in (\mathbb{Z}/p\mathbb{Z})^\times$ with $a^2 \equiv -1 \pmod p$. The subgroup corresponding to the  subfield $K_q \subset L$ is given by 
\begin{equation*}
H_q = \{1, -1, a, -a\} / \{\pm 1\} \cong \{1, a\}.
\end{equation*}
\begin{definition}\label{def}
We define the following cyclotomic units in $\mathcal{O}_{K_q}$:
\begin{equation*}
    \varepsilon_q:= (1+\zeta_p)(1+\zeta_p^{-1})(1+\zeta_p^a)(1+\zeta_p^{-a}) = (2+e_1)(2+e_a) \in \mathcal{O}_{K_q}.
\end{equation*}
\end{definition}

Let $T_q \subset \{1, 2, \dots, 2q\}$ be a set of coset representatives for $G / H_q$. The $q$ distinct Galois conjugates of $\varepsilon_q$ in $K_q$ are given by $(2+e_k)(2+e_{ak})$ for $k \in T_q$. Expanding the product, we obtain
\begin{equation}\label{S_1}
\Tr_{K_q/\mathbb{Q}}(\varepsilon_q) = \sum_{k \in T_q} (2+e_k)(2+e_{ak}) = 4q + 2 \sum_{k \in T_q} (e_k + e_{ak}) + \sum_{k \in T_q} e_k e_{ak}.
\end{equation}
Note that the index set $\{k, ak \pmod p \mid k \in T_q\}$ forms a complete set of representatives for $G$, which corresponds to $\{1, 2, \dots, 2q\}$. Thus, the sum reduces to $2 \sum_{j=1}^{2q} e_j = 2(-1) = -2$.

Let $S := \sum_{k \in T_q} e_k e_{ak}$. We next evaluate $S$ explicitly:

\begin{lemma}
$S =\sum_{k \in T_q} e_k e_{ak}= -1$.
\end{lemma}
\begin{proof}
Expanding $e_k e_{ak}$ into roots of unity yields
\[
e_k e_{ak} = (\zeta_p^k + \zeta_p^{-k})(\zeta_p^{ak} + \zeta_p^{-ak}) = \zeta_p^{k(1+a)} + \zeta_p^{-k(1+a)} + \zeta_p^{k(1-a)} + \zeta_p^{-k(1-a)}.
\]
The exponent set for a fixed $k$ is $k \{ \pm(1+a), \pm(1-a) \}$. Since $a^2 \equiv -1 \pmod p$, we have the algebraic congruence
\begin{equation*}
\frac{1-a}{1+a} \equiv \frac{(1-a)^2}{1-a^2} \equiv \frac{-2a}{2} \equiv -a \pmod p.
\end{equation*}
This implies that the exponent set is equal to $k(1+a)\{\pm 1, \pm a\}$. As $k$ runs through $T_q$, the set of all $4q$ exponents $\{ k h (1+a) \pmod p \mid k \in T_q, h \in \{\pm 1, \pm a\} \}$ covers every non-zero residue class in $(\mathbb{Z}/p\mathbb{Z})^\times$ exactly once. Therefore
$
S = \sum_{x=1}^{p-1} \zeta_p^x = -1.
$
\end{proof}

Substituting $S = -1$ back into Equation \eqref{S_1} yields
$
\Tr_{K_q/\mathbb{Q}}(\varepsilon_q) = 4q - 2 + (-1) = p - 4.
$ Combining the preceding results, we obtain the following theorem.

\begin{theorem}\label{p-4}
The explicitly constructed totally positive unit $\varepsilon_q \in U_{K_q}^2$ satisfies $
 \Tr_{K_q/\mathbb{Q}}(\varepsilon_q) = p - 4$.
Hence,  when $U_L^+=U_L^2$, the lower bound for the trace of non-trivial totally positive units in $K_q$ is sharp.
\end{theorem}
\begin{proof}
Let $m = (p+1)/2$, so $2m \equiv 1 \pmod p$. The cyclotomic unit as 
$
1 + \zeta_p = 1 + \zeta_p^{2m} = \zeta_p^m (\zeta_p^{-m} + \zeta_p^m).
$
Let $\varepsilon=(1+\zeta_p)(1+\zeta_p^{-1})$.
Substituting this into $\varepsilon$, we obtain
\begin{equation*}
\varepsilon= (1+\zeta_p)(1+\zeta_p^{-1}) = \zeta_p^m \zeta_p^{-m} (\zeta_p^m + \zeta_p^{-m})^2 = (e_m)^2.
\end{equation*}
Since $e_m \in \mathcal{O}_L$, $\varepsilon$ is a square in $\mathcal{O}_L$.  Observe that $\varepsilon_q = \varepsilon \cdot \sigma_a(\varepsilon) = (e_m)^2 (e_{am})^2 = (e_m e_{am})^2$. We must show $\eta_q := e_m e_{am}\in K_q$. Notice that the Galois group $\Gal(L/K_q)$ is generated by $\sigma_a$, applying $\sigma_a$ to $\eta_q$ gives
$
\sigma_a(e_m e_{am}) = e_{am} e_{a^2 m}.
$
Since $a^2 \equiv -1 \pmod p$ and $e_{-x} = e_x$ for all $x$, we have $e_{a^2 m} = e_{-m} = e_m$. Thus,
$
\sigma_a(\eta_q) = e_{am} e_m = \eta_q.
$
This shows that $\eta_q$ is fixed by $\Gal(L/K_q)$,  $\eta_q \in \mathcal{O}_{K_q}$. Therefore, $\varepsilon_q = \eta_q^2$ is a square unit in $U_{K_q}^2$ and is totally positive. 
\end{proof}

\section{Spherical fusion category of prime global dimension}\label{sectionfusion}
In this section, we apply the trace inequality to the classification of spherical fusion categories of global dimensions being a prime.


Let $\C$ be a spherical fusion category of global dimension $p$, where $p$ is a prime, then $f_\chi=p\cdot u_{f_\chi}$ for any formal codegree $f_\chi$ of $\C$ \cite[Lemma 5.1]{Schopieray2021}, where $u_{f_\chi}\in\mathbb{Q}(\zeta_p)^+$ is an algebraic unit. In this particular case, then Equation (\ref{generalclasseq})  implies (see also \cite[Proposition 5.4]{Schopieray2021})
\begin{align}\label{classeqt}
\sum_{\chi}\chi(1)u_{f_\chi}^{-1}=p.
\end{align}
By using the classical Siegel's trace theorem \cite{Siegel1945},  \(\mathcal{C}\) is pointed or \(p = 5\) if  \(p = 2q + 1\) \cite[Theorem 5.5]{Schopieray2021}  , where \( q\geq 2\) is also a prime. For $p=5$, the structure of $\C$ is known by \cite[Example 5.1.2]{Ostrik2019}.

It was claimed that spherical fusion category of global dimension $13$ is also pointed in \cite[Example 5.7]{Schopieray2021}. However,  there does exist a case that is not excluded by the method of \cite{Schopieray2021}. Indeed, by Definition \ref{def} and Theorem \ref{p-4},  when $p=13$, $5^2 \equiv -1 \pmod{13}$, there exist $\varepsilon_3=(2+e_1)(2+e_5)  \in U_{K_3}^+$. The minimal polynomial of this element is $T^3-9T^2+14T-1$, therefore, $ \Tr_{K_3/\mathbb{Q}}(\varepsilon_3) = 9$.
\begin{lemma}\label{13ranklemma}
 Let $\C$ be a spherical fusion category of global dimension $13$. If $\C$ is not pointed, then $\mathrm{rank}(\C)=7$ and $\mathrm{FPdim}(\C)=13\alpha$, where $\alpha$ is the largest root of
$
T^3-14T^2+9T-1.
$  
\end{lemma}
\begin{proof}
Assume there exists a non-pointed spherical fusion category of global dimension $13$. Let  $f$ be a formal codegree such that $f\neq 13$, denote $d_f:=[\mathbb{Q}(f):\mathbb{Q}]$, it follows from 
\cite[Example 5.7]{Schopieray2021} that $d_f=3$. By using the program GAP \cite{GAP}, we can know that there exists a unique irreducible polynomial  $T^3-14T^2+9T-1$ that passes the criterion given in \cite[Example 5.7]{Schopieray2021}. Since $\C$ is not pointed, $\mathrm{FPdim}(\C)\neq 13$ \cite[Corollary 8.30]{ENO2005}, we know $\mathrm{FPdim}(\C)=13\alpha$, where $\alpha$ is the largest root of $T^3-14T^2+9T-1$, then the Equation \eqref{classeqt}
 shows the formal codegrees of $\C$ are $13, 13,13,13,\mathrm{FPdim}(\C)$ and its two Galois conjugates. Notice that the Galois conjugates of $\dim(-)$ and $\mathrm{FPdim}(-)$ are  six different ring homomorphisms of $\mathrm{Gr}(\C)\otimes_\Z\mathbb{C}$, therefore, $\C$ must has commutative fusion rules and $\mathrm{rank}(\C)=7$. 
 \end{proof}
 Let 
$
L=\mathbb Q(\zeta_{13})^+$, define
$
\eta_0:=e_1+e_5,
\eta_1:=e_2+e_3,
\eta_2:=e_4+e_6.
$
The identities
$
e_ie_j=e_{i+j}+e_{i-j},
1+e_1+\cdots+e_6=0
$
give
\begin{equation}\label{(1)}
\eta_0+\eta_1+\eta_2=-1,
\quad
\eta_j^2=4+\eta_{j+1}+2\eta_{j+2},
\end{equation}
where the subscripts on the \(\eta_j\) are taken modulo \(3\).
In particular,
\begin{equation}\label{(2)}
\eta_0^2+\eta_1^2+\eta_2^2=9,
\quad
\prod_{j=0}^{2}(T-\eta_j)=T^3+T^2-4T+1, \quad \prod_{j=0}^{2}(T-\eta_j^{-2})
=T^3-14T^2+9T-1,
\end{equation}
thus \(\eta_0\) is a unit, with
$
\eta_0^{-1}=4-\eta_0-\eta_0^2$, and $\alpha=\eta_0^{-2}$.
\begin{proposition}\label{13dimension}
Let $\C$ be a spherical fusion category of global dimension $13$. If $\C$ is not pointed, then the Frobenius-Perron dimensions of simple objects of $\C$ are $1,u,u,u,u,v,w$, where $
u=\eta_0^{-1},
v=\eta_1\eta_0^{-1},
w=-\eta_2\eta_0^{-1}$.
\end{proposition}
\begin{proof}
By Lemma \ref{13ranklemma}, $\mathrm{rank}(\C)=7$ and $\mathrm{FPdim}(\C)=13\alpha$. Write the simple objects as
$
\mathbf1,V_1,\ldots,V_6
$, and  let $
d_i=\operatorname{FPdim}(V_i)$, then we have
$
\mathrm{FPdim}(\C)=1+\sum_{i=1}^{6}d_i^2=13\eta_0^{-2}$. Therefore, the positive algebraic integers $
b_i=\eta_0d_i$
satisfy 
\begin{equation}\label{(4)}
\sum_{i=1}^{6}b_i^2=13-\eta_0^2.
\end{equation}

By \cite[Corollary 2.2.1]{Davis1969}, $\{e_1,e_2,e_3,e_4,e_5,e_6\}$ is a \(\mathbb{Z}\)-basis of \(\mathcal{O}_L\). For \(b=\sum_{j=1}^{6}n_je_j\in\mathcal O_L\) with $n_j\in\mathbb{Z}$, by Equation \eqref{trace_L_formula}, after expanding and simplifying, we have
\begin{equation}\label{(5)}
\Tr_{L/\mathbb{Q}}(b^2)
=13\sum_{j=1}^{6}n_j^2
  -2\left(\sum_{j=1}^{6}n_j\right)^2
=\sum_{j=1}^{6}n_j^2
  +2\sum_{1\leq i<j\leq6}(n_i-n_j)^2.
\end{equation}
If all \(n_j=m\), then \(b=-m\) and \(\Tr_{L/\mathbb{Q}}(b^2)=6m^2\).
Otherwise, 
there are at least five pairs of unequal coordinates, so $\Tr_{L/\mathbb{Q}}(b^2)\geq11$.
Therefore, for every nonzero element \(b\in\mathcal O_L\),
\begin{equation}\label{(6)}
\Tr_{L/\mathbb{Q}}(b^2)=6\quad\text{if }b=\pm1,
\qquad
\Tr_{L/\mathbb{Q}}(b^2)\geq11\quad\text{otherwise}.
\end{equation}
Since \(\Tr_{L/\mathbb{Q}}(\eta_0^2)=18\), 
$\sum_{j=1}^{6}\Tr_{L/\mathbb{Q}}(b_j^2)=60$ by Equation \eqref{(4)}, so \(\Tr_{L/\mathbb{Q}}(b_j^2)\leq30\) for all $j$.

Let $\tau\in \mathrm{Gal}(L/\Q)$ such that $\tau(\zeta_{13})=\zeta_{13}^5$, then $\tau(e_1)=e_5$, $\tau(e_2)=e_3$, $\tau(e_4)= e_6$ and $\tau(\eta_0)=\eta_0$.
 Hence, we obtain
\[
1+\sum_{i=1}^6\tau(d_i)^2
=13\eta_0^{-2}
=1+\sum_{i=1}^6 d_i^2.\]
Note $|\tau(d_j)|\leq d_j$, as the Frobenius-Perron dimension of object are the largest among its Galois conjugates \cite[Proposition 3.3.6]{EGNO},  we have $
\tau(b_j)=\pm b_j$ for all $j$. 
If \(\tau(b_i)=-b_i\) for some $i$, its coefficient vector in the integral
basis of $\mathcal{O}_L$ has the form $
(r,s,-s,t,-r,-t)$,
so Equation \eqref{(5)} gives
\begin{equation*}
\Tr_{L/\mathbb{Q}}(b_i^2)=26(r^2+s^2+t^2),
\end{equation*}
 which would force \(\Tr_{L/\mathbb{Q}}(b_i^2)=26\) as \(\Tr_{L/\mathbb{Q}}(b_i^2)\leq30\).
The other five traces would then sum to \(34\), which then contradicts
Inequality \eqref{(6)}.
Hence,
$
\tau(b_j)=b_j$ for all $1\leq j\leq 6$.

Meanwhile, for any $b\in\mathcal{O}_L$, if $\tau(b)=b$, then 
$b=r\eta_0+s\eta_1+t\eta_2$, $r,s,t\in\mathbb Z$, and 
\begin{equation}\label{(10)}
\Tr_{L/\mathbb{Q}}(b^2)
=
2\left(
r^2+s^2+t^2
+4\bigl((r-s)^2+(s-t)^2+(t-r)^2\bigr)
\right).
\end{equation}
If $\tau(b)=b$, $b\neq \pm1$ and  \(r,s,t\) are not all equal, we have
\begin{equation}\label{(11)}
\Tr_{L/\mathbb{Q}}(b^2)\geq18,
\end{equation}
with equality occurs precisely for $
b\in\{\pm\eta_0,\pm\eta_1,\pm\eta_2\}$, equivalently
\(r^2+s^2+t^2=1\). Let \(k\) be the number of \(b_j\) different from $1$.
Inequality \eqref{(11)} implies $
60\geq6(6-k)+18k$,
so \(k\leq2\).
The case \(k=0\) would give total trace \(36\).
If \(k=1\), the remaining element would have trace \(30\).
However, by Equation \eqref{(10)}, this would imply $
13(r^2+s^2+t^2)-4(r+s+t)^2=15$,
 hence $(r+s+t)^2\equiv6\pmod{13}$,
which is impossible as \(6\) is not a quadratic residue
modulo \(13\).
Thus \(k=2\), and traces of both these two exceptional elements (denote by $b,c$) are  \(18\). Then the preceding argument shows
$b,c\in\{\pm\eta_0,\pm\eta_1,\pm\eta_2\}$.
Since all \(b_j\) are positive,  a direct computation shows that 
$b=\eta_1,c=-\eta_2
$. Denote
$
u=\eta_0^{-1},
v=\eta_1\eta_0^{-1},
w=-\eta_2\eta_0^{-1}, 
$
the Frobenius-Perron dimensions of the simple objects are $1,u,u,u,u,v,w$, as claimed.
\end{proof}

Since $v=\eta_1\eta_0^{-1}$ and $
w=-\eta_2\eta_0^{-1}$, we know $v=u^2-2u-1$ and $w=u^2-u$. 
\begin{theorem}\label{13theorem}
There is no spherical fusion category \(\mathcal C\) satisfying
$
\operatorname{rank}(\mathcal C)=7, \dim(\mathcal C)=13$ and $\operatorname{FPdim}(\mathcal C)=13\alpha$. That is, spherical fusion categories of global dimension $13$ are pointed. 
\end{theorem}
\begin{proof}
By Proposition \ref{13dimension}, the Frobenius-Perron dimensions of simple objects are $1,u,u,u,u,v,w$. Without loss of generality, assume $\text{FPdim}(V_i)=u$ for $1\leq i\leq 4$, $\text{FPdim}(V_5)=v$ and $\text{FPdim}(V_6)=w$. Let $g(T)=T^3-4T^2+T+1$, then $g(T)$ is irreducible over $\mathbb Q$ and $g(u)=0$. Let $u_1,u_2,u_3$ be the three distinct
roots of $g(T)$, and let $\phi_j$ be the homomorphisms that are conjugate to $\mathrm{FPdim}(-)$, then for any simple object $V$ of $\C$, we have $\phi_j(V^*)=\phi_j(V)$ for $1\leq j\leq 3$, also  $\phi_j\neq \dim(-)$ as $\C$ is  not a pseudo-unitary fusion category. 

Let 
$x_i=\dim(V_i)$, $1\leq i\leq 4$, $ y=\dim(V_5),
z=\dim(V_6),
S=\sum_{i=1}^{4}x_i$.
Then \cite[Lemma 8.14.1]{EGNO} shows that $\sum_{V\in\mathcal{O}(\C)}\dim(V)\phi_j(V^*)=0$ for $j=1,2,3$, that is,
\begin{equation*}
1+Su_j+y(u_j^2-2u_j-1)+z(u_j^2-u_j)=0,\quad j=1,2,3.
\end{equation*}
Thus the polynomial
$h(T):=
(y+z)T^2+(S-2y-z)T+(1-y)
$
vanishes at the three distinct roots $u_j$ of $g(T)$.
The degree  of $h(T)$ is at most two, so it is identically zero.
Consequently,
\begin{equation*}
y=1,\quad z=-1,\quad S=1.
\end{equation*}
Meanwhile,
$\dim(\C)=1+\sum_{i=1}^{4}x_i^2+y^2+z^2,
$
hence $\sum_{i=1}^4 x_i^2=10$. Therefore,  
\begin{equation*}
\sum_{1\leq i<j\leq4}x_ix_j
=
\frac{\left(\sum_{i=1}^4 x_i\right)^2-\sum_{i=1}^4 x_i^2}{2}
=-\frac{9}{2},
\end{equation*}
which is impossible, as the left-hand side is an algebraic integer. Therefore, spherical fusion category of global dimension $13$ is pointed.
\end{proof}

We now complete the proof of Theorem~\ref{maintheorem2} by using Corollary \ref{1-2} and Theorem \ref{13theorem}.

\begin{proof}[Proof of Theorem \ref{maintheorem2}]
Let $\mathcal{C}$ be a spherical fusion category of global dimension $p=4q+1$. Assume that $2$ is a primitive root modulo $q$, or
 $p=4q+1<2^{22}$ and $p\neq 29$ below. If $q=3$, this is proved in Theorem \ref{13theorem}. By assumption, we have $q>7$. 

It is proved in \cite[Lemma 5.1, Corollary 5.3]{Schopieray2021} that $\mathrm{FPdim}(\mathcal{C})=p\cdot u_{\mathrm{FPdim}(\mathcal{C})}$, where $u_{\mathrm{FPdim}(\mathcal{C})}\in \mathbb{Q}(\zeta_p)^+$ is a totally positive  algebraic unit, then  $d_{\mathrm{FPdim(\mathcal{C})}} \in\{1,2,q,2q\}$.  If  $d_{\mathrm{FPdim(\mathcal{C})}} =1$, then $\mathcal{C}$ is pointed, we are done.
It follows from \cite[Proposition 3.6]{Yu2023} that   $d_{\mathrm{FPdim(\mathcal{C})}} \neq 2,2q$. Assume  $d_{\mathrm{FPdim(\mathcal{C})}} =q$ below. Notice that all the formal codegrees of $\mathcal{C}$ belong to the field $\mathbb{Q}(\dim(X),X\in\mathcal{O}(\mathcal{C})
)$ by \cite[Corollary 2.15]{Ostrik2015},  so the ring homomorphism $\dim(-)$   has at least $q$
Galois conjugates, which means that $\C$ has at least $q$ formal codegrees being $p$, thus the Equation  \eqref{classeqt} implies 
\begin{equation*}
\Tr_{\mathbb{Q}(u_{\mathrm{FPdim(\mathcal{C})}})/\mathbb{Q}}(u^{-1}_{\mathrm{FPdim(\mathcal{C})}})\leq p-q.
\end{equation*}
However, $\Tr_{\mathbb{Q}(u_{\mathrm{FPdim(\mathcal{C})}})/\mathbb{Q}}(u^{-1}_{\mathrm{FPdim(\mathcal{C})}}) \geq p-4$ by Corollary \ref{1-2}, which is a contradiction as $q > 7$. 
\end{proof}

\begin{remark}
    We remark that the method in Proposition \ref{13dimension} can't be applied to the classification of spherical fusion category $\C$ of global dimension $29$. Indeed, if $\C$ is not pointed, it follows from Remark \ref{remarkp=4q+1} that  $
    \Tr_{\mathbb{Q}(u_{\mathrm{FPdim(\mathcal{C})}})/\mathbb{Q}}(u^{-1}_{\mathrm{FPdim(\mathcal{C})}}) \geq 16$, so new methods are definitely required to classify spherical fusion categories of prime global dimension.
\end{remark}

\subsection*{Acknowledgements}
The authors thank A. Schopieray for communications on spherical fusion category of prime global dimension and comments on an early draft, which improve this manuscript substantially.  Z. Yu  is supported
by NSFC (no.12571041) and Qinglan Project of Yangzhou University.
The authors acknowledge the use of AI as an auxiliary tool during the preparation of this manuscript, including for mathematical discussion and editorial assistance. All
mathematical content was independently verified by the authors, who take full responsibility for
the final manuscript.

\author{Yue Meng\\ \thanks{Email:\,DX120250050@stu.yzu.edu.cn}
\\{\small School  of Mathematics,  Yangzhou University,
Yangzhou 225002, China}
}\\

\noindent \author{Zhiqiang Yu\\ \thanks{Email:\,zhiqyumath@yzu.edu.cn}
\\{\small School  of Mathematics,  Yangzhou University,
Yangzhou 225002, China}
}

\end{document}